\documentclass{amsart}
\usepackage{amsmath, amsthm, amscd, amsfonts, amssymb, graphicx, color}
\usepackage{enumerate}
\usepackage[bookmarksnumbered, colorlinks, plainpages]{hyperref}
\usepackage{tikz}
\usepackage{comment}

\newtheorem{theorem}{Theorem}[section]
\newtheorem{lemma}[theorem]{Lemma}
\newtheorem{proposition}[theorem]{Proposition}

\newtheorem{corollary}[theorem]{Corollary}
\newtheorem{definition}[theorem]{Definition}
\newtheorem{example}[theorem]{Example}
\newtheorem{remark}[theorem]{Remark}

\excludecomment{mysection}

\DeclareMathOperator{\co}{co}

\begin{document}
	
\title{Centre of a compact convex set}
\author{Anil Kumar Karn}
	
\address{School of Mathematical Sciences, National Institute of Science Education and Research Bhubaneswar, An OCC of Homi Bhabha National Institute, P.O. - Jatni, District - Khurda, Odisha - 752050, India.}

\email{\textcolor[rgb]{0.00,0.00,0.84}{anilkarn@niser.ac.in}}

\subjclass[2020]{Primary: 46B40; Secondary: 46B20.}
	
\keywords{Lead points of a convex set, centre of a convex set, Property (S), tracial absolute order unit space, absolutely central base normed space.}
	
\begin{abstract}
	We introduce the notion centre of a convex set and study the space of continuous affine functions on a compact convex set with a centre. We show that these spaces are precisely the dual of a base normed space in which the underlying base has a (unique) centre. We also characterize the corresponding base norm space. We obtain a condition on a compact, balanced, convex subset of a locally convex space so that the corresponding space of continuous affine functions on the convex set is an absolute order unit space. Similarly, we characterize a condition on the base with a centre of a base normed space so that the latter becomes an absolutely base normed space.
\end{abstract}

\maketitle 

\section{Introduction} 

Let $K$ be a compact, convex subset of a locally convex space $X$ and let $A_{\mathbb{R}}(K)$ denote the space of real valued, continuous and affine functions defined on $K$. Then $A_{\mathbb{R}}(K)$ is a complete order unit space and the state space of $A_{\mathbb{R}}(K)$ is affine homeomorphic to $K$. Conversely, if $(V, e)$ is a complete order unit space with its state space $S(V)$, then $V$ is unitally order isomorphic to $A_{\mathbb{R}}(S(V))$. The self-adjoint part of an operator system (that is, a self-adjoint subspace containing identity) in a unital C$^*$-algebra $A$ is an order unit space. Therefore, the above said discussion describes Kadison's functional representation theorem.

Further, if $K$ and $L$ are compact convex sets in suitable locally convex spaces, then $K$ is affine homeomorphic to $L$ if and only if $A_{\mathbb{R}}(K)$ is unitally order isomorphic to $A_{\mathbb{R}}(L)$. In other words, there exists a bijective correspondence between the class of compact convex sets of locally convex spaces and the class of complete order unit spaces. 

In this paper, we discuss an intrinsic characterization of compact convex sets which are affine homeomorphic to balanced, convex, compact sets. We also discuss the corresponding subclass of complete order unit spaces. We introduce the notion of a centre of a convex set and prove that a compact convex set $K$ in a locally convex space $X$ is affine homeomorphic to a balanced, convex, compact set $L$ in another locally convex space $Y$ if and only if $K$ has a (unique) centre. 

We study the properties of the space $A(B)$ of continuous affine functions on a compact convex set $B$ with a centre $b_0$ in a locally convex space $X$. We prove that $A(B)$ is unitally order isomorphic to $(V^{(\cdot)}, e)$ for some Banach space $V$. Here $(V^{(\cdot)}, e)$ is the order unit space obtained by adjoining an order unit to $V$ (Theorem \ref{day}). Next, we prove that the base normed space described in Proposition \ref{adjbase} is precisely a base normed space in which the corresponding base has a (unique) centre. We also discuss the order unit spaces whose state space have a centre. 

In \cite{K21}, the author discussed an order theoretic generalization of spin factors obtained by adjoining an order unit to a normed linear space. Let $V$ be a normed linear space and let $(V^{(\cdot)}, e)$ is the order unit space obtained by adjoining an order unit to $V$. He proved that there is a canonical absolute value defined in $V^{(\cdot)}$ and that $(V^{(\cdot)}, e)$ becomes an absolute order unit space if and only if $V$ is strictly convex \cite[Theorem 2.14]{K21}. In the present paper, we discuss an absolute value in a base normed space in which the base has a centre and find a condition under which the space becomes an absolutely base normed space (Theorem \ref{bstrict}).

A summary of the paper is as follows.

In Section 2, we recall the adjoining of an order unit to a normed linear space and the notion of lead points of a convex set and introduce the notion of a centre of a convex set. We study some related properties. 

In Section 3, we consider the space of continuous affine functions on a compact convex set with a centre and describe its order and norm structures. We prove that this space can be characterized as one obtained by adjoining an order unit to a Banach space. 

In Section 4,we study the norm and order structure of a base norm space  in which the base has a (unique) centre. We produce an example to show that the base of every base normed space need not have a centre.

In Section 5, we consider the order unit spaces having a central state space. We obtain a characterization for a state to be a centre of the state space. We study the norm and order structure an order unit space having a central state and prove that such a space is precisely obtained by adjoining an order unit to a normed linear space. Let us recall that such a space becomes an absolute order unit space if the underlying normed space is strictly convex \cite[Theorem 3.14]{K21}. We introduce the Property (S) on a compact, balanced, convex subset of a real locally convex space and prove that an order unit space having a central state is an absolute order unit space if and only if the corresponding state space satisfies Property (S).
 
In Section 6, we describe an absolute value on a base normed space in which the base has a centre. This absolute value arises naturally. We characterize the condition on the underlying base under which the space becomes an absolutely ordered base normed space.

\section{Centre of a convex set} 

\subsection{Adjoining an order unit to a normed linear space.}
The following construction has been adopted from \cite[1.6.1]{Jam70} and is apparently due to M. M. Day. 
\begin{theorem}[M. M. Day]\label{day}
	Let $(V, \Vert \cdot \Vert)$ be a real normed linear space. Consider $V^{(\cdot)} := V \times \mathbb{R}$ and define 
	$$V^{(\cdot)+} := \lbrace (v, \alpha): \Vert v \Vert \le \alpha \rbrace.$$ 
	Then $(V^{(\cdot)}, V^{(\cdot)+})$ becomes a real ordered space such that $V^{(\cdot)+}$ is proper, generating and Archimedean. Also, $e = (0, 1) \in V^{(\cdot)+}$ is an order unit for $V^{(\cdot)}$ so that $(V^{(\cdot)}, e)$ becomes an order unit space.  The corresponding order unit norm is given by 
	$$\Vert (v, \alpha) \Vert_e = \Vert v \Vert + \vert \alpha \vert$$
	for all $(v, \alpha) \in V^{(\cdot)}$. In particular, $(V^{(\cdot)}, e)$ is isometrically isomorphic to $V \oplus_1 \mathbb{R}$. Thus $V$, identified with $\lbrace (v, 0): v \in V \rbrace$, can be identified as a closed subspace of $V^{(\cdot)}$. Further, $V^{(\cdot)}$ is complete if and only if so is $V$.
\end{theorem}
The following result can proved in a routine way. 
\begin{proposition}
	Let $V$ be a real normed linear space and consider the corresponding order unit space $(V^{(\cdot)}, e)$. Then the dual of $(V^{(\cdot)}, e)$ is a base normed space isometrically isomorphic to $(V^{*(\cdot)}, S(V^{(\cdot)}))$ where 
	$$S(V^{(\cdot)}) := \lbrace (f, 1): f \in V^* ~ \textrm{and} ~ \Vert f \Vert \le 1 \rbrace$$ 
	is the state space of $(V^{(\cdot)}, e)$ as well as the base of $V^{*(\cdot)+}$ which determine the dual norm as the base norm on $V^{*(\cdot)}$. 
\end{proposition} 
We show that a similar construction is also available on $V^{(\cdot)}$ itself albeit with a different norm (and the same order structure). 
\begin{proposition}\label{adjbase}
	Let $V$ be a real normed linear space and consider the corresponding ordered vector space $(V^{(\cdot)}, V^{(\cdot)+})$. Put 
	$$B^{(\cdot)} = \lbrace (v, 1): \Vert v \Vert \le 1 \rbrace.$$
	Then $B^{(\cdot)}$ is a base for $V^{(\cdot)+}$ such that $(V^{(\cdot)}, B^{(\cdot)+})$ is a base normed space. The corresponding base norm is given by 
	$$\Vert (v, \alpha) \Vert_B = \max \lbrace \Vert v \Vert, \vert \alpha \vert \rbrace$$ 
	for all $(v, \alpha) \in V^{(\cdot)}$. In particular, $(V^{(\cdot)}, B^{(\cdot)})$ is isometrically isomorphic to $V \oplus_{\infty} \mathbb{R}$.  Thus $V$, identified with $\lbrace (v, 0): v \in V \rbrace$, can be identified as a closed subspace of $V^{(\cdot)}$. Further, $V^{(\cdot)}$ is complete if and only if so is $V$.
\end{proposition}
\begin{proof}
	It suffices to prove that $B^{(\cdot)}$ is a base for $V^{(\cdot)+}$ such that 
	$$\co \left(B^{(\cdot)} \cup - B^{(\cdot)}\right) = \lbrace (x, \alpha): \Vert x \Vert \le 1, \vert \alpha \vert \le 1 \rbrace.$$ 
	Note that $B^{(\cdot)}$ is convex. If $(v, \alpha) \in V^{(\cdot)+}$, then $\Vert v \Vert \le \alpha$. So if $(v, \alpha) \ne (0, 0)$, then $\alpha > 0$. Now letting $x_0 = \alpha^{-1} x$, we get that $(x_0, 1) \in B$ and $(x, \alpha) = \alpha (x_0, 1)$ is a unique representation. We show that 
	$$\co \left(B^{(\cdot)} \cup - B^{(\cdot)}\right) = \lbrace (x, \alpha): \Vert x \Vert \le 1, \vert \alpha \vert \le 1 \rbrace.$$ 
	Let $(x, \alpha) \in \co \left(B^{(\cdot)} \cup - B^{(\cdot)}\right)$, say $(x, \alpha) = \lambda (x_1, 1) - (1 - \lambda) (x_2, 1)$ for some $x_1, x_2 \in V$ with $\Vert x_1 \Vert \le 1$ and $\Vert x_1 \Vert \le 1$ and $\lambda \in [0, 1]$. Then $x = \lambda x_1 - (1 - \lambda) x_2$ and $\alpha = \lambda - (1 - \lambda) = 2 \lambda - 1$. Thus $\Vert x \Vert \le 1$ and $\vert \alpha \vert \le 1$. Conversely, we assume that $x \in V$ and $\alpha \in \mathbb{R}$ be such that $\Vert x \Vert \le 1$ and $\vert \alpha \vert \le 1$. If $\Vert x \Vert \le \vert \alpha \vert$, then $(\alpha^{-1}x, 1) \in B^{(\cdot)}$ so that $(x, \alpha) = \alpha (\alpha^{-1} x, 1) \in \co \left(B^{(\cdot)} \cup - B^{(\cdot)}\right)$. So we assume that $\vert \alpha \vert < \Vert x \Vert$. Put $y = \Vert x \Vert^{-1} x$. Then $(y, 1), (- y, 1) \in B^{(\cdot)}$. Let $2 \lambda = \Vert x \Vert + \alpha$ and $2 \mu = \Vert x \Vert - \alpha$. Then $\lambda, \mu \in [0, 1]$ with $\lambda + \mu = \Vert x \Vert$ and $\lambda - \mu = \alpha$. Thus 
	$$\lambda (y, 1) - \mu (- y, 1) = ((\lambda + \mu) y, \lambda - \mu) = (x, \alpha)$$ 
	so that $(x, \alpha) \in \co \left(B^{(\cdot)} \cup - B^{(\cdot)}\right)$. 
\end{proof}

In this section, we shall obtain a geometric description of $B^{(\cdot)}$. First, we recall the following notion introduced in \cite{GK20}. 

\subsection{Lead points of a convex set}
\begin{definition} \cite{GK20}
	Let $C$ be a convex subset of a real vector space $X$ with $0 \in C$. An element $x \in C$ is called a \emph{lead point} of $C$, if for any $y \in C$ and $\lambda \in [0, 1]$ with $x = \lambda y$, we have $\lambda = 1$ and $y = x$. The set of all lead points of $C$ is denoted by $Lead(C)$. 
\end{definition} 

The following result was essentially proved in \cite[Proposition 3.2]{GK20}.
\begin{lemma}\label{l1}
	Let $C$ be a non-empty, radially compact, convex subset of a real vector space $X$ containing $0$. Then for each $x \in C$ with $x \ne 0$, there exist a unique $x_1 \in Lead(C)$ and a unique $\lambda \in (0, 1]$ such that $x = \lambda x_1$. (A non-empty subset $S$ of $X$ is called \emph{radially compact}, if for any $x \in X$ with $x \ne 0$, the set $\lbrace \alpha \in \mathbb{R}: \alpha x \in S \rbrace$ is compact in $\mathbb{R}$.)
\end{lemma}
We extend this result to a simple but interesting characterization of a norm on a real vector space.
\begin{theorem}\label{nls}
	Let $C_0$ be a non-empty, radially compact, absolutely convex subset of a real vector space $X$. Let $X_0$ be the linear span of $C_0$ so that $X_0 = \cup_{n=1}^{\infty} n C_0$. 
	\begin{enumerate}
		\item Then for each non-zero $x \in X_0$, there exists a unique $c \in Lead(C_0)$ and a unique $\alpha > 0$ such that $x = \alpha c$. 
		\item Let us write $\alpha := r(x)$ and put $r(0) = 0$. Then $r: X_0 \to \mathbb{R}^+$ determines a norm on $X_0$ such that $C_0$ is the closed unit ball.
	\end{enumerate}
\end{theorem}
\begin{proof}
	(1): Fix $x \in X_0$, $x \ne 0$. As $C_0$ absorbing in $X_0$, there exists a non-zero $k \in \mathbb{R}$ such that $k x \in C_0$. Since $C_0$ is absolutely convex, we may take $k > 0$. As $C_0$ is radially compact, the set $S = \lbrace \lambda \in \mathbb{R}: \lambda x \in C_0 \rbrace$ is a compact set in $\mathbb{R}$. Also, $\sup S = \alpha_0 > 0$ with $\alpha_0 x \in C_0$. Put $\alpha_0 x = c$. We show that $c \in Lead(C_0)$. By Lemma \ref{l1}, we have $c = \lambda c_1$ for some $c_1 \in Lead(C_0)$ and $\lambda \in (0, 1]$ so that $\alpha \lambda^{-1} \in S$. As $\alpha_0 = \sup S$, we get that $\lambda = 1$ and consequently, $c \in Lead(C_0)$. Next, let $x = \beta d$ for some $d \in Lead(C_0)$ and $\beta > 0$. Then $\beta^{-1} \le \alpha_0$. Also, we have $c = \alpha_0 \beta d$. Thus by the definition of $Lead(C_0)$, we get $c = d$ and $\beta = \alpha_0^{-1}$. 
	
	(2) We note that $r(x) \le 1$ for all $x \in C_0$ and that $r(x) = 1$ if and only if $x \in Lead(C_0)$. Also, for $x \in X_0$, $x \ne 0$ we have $r(x) > 0$. Let $x \in X_0$ and $\alpha \in \mathbb{R}$. Without any loss of generality, we assume that $x \ne 0$ and $\alpha \ne 0$. Let $x = \lambda c$ be the unique representation with $c \in Lead(C_0)$. Then $r(x) = \lambda$. If $\alpha > 0$, then $\alpha x = \alpha \lambda c$. Thus by the uniqueness of representation, we get that $r(\alpha x) = \alpha \lambda = \alpha r(x)$. For $\alpha < 0$, it suffices to prove that $- c \in Lead(C_0)$ whenever $c \in Lead(C_0)$. Let $c \in Lead(C_0)$. Then $- c \in C_0$ as $C_0$ is balanced. Assume that $- c = \beta d$ for some $d \in Lead(C_0)$. Then $c = \beta (- d)$ with $- d \in C_0$. Now by the definition of $Lead(C_0)$ we get $- d = c$ and $\beta = 1$. Thus $- c \in Lead(C_0)$. Finally, we show that $r$ is sub-additive. For this, let $x, y \in X_0$. Without any loss of generality, we assume that $x \ne 0$ and $y \ne 0$. Let $x = \alpha c$ and $y = \beta d$ be the unique representations with $c, d \in Lead(C_0)$ and $\alpha > 0$, $\beta > 0$. Then $r(x) = \alpha$ and $r(y) = \beta$. Now $x + y = (\alpha + \beta) z$ where $z = \left( \frac{\alpha}{\alpha + \beta} c + \frac{\beta}{\alpha + \beta} d \right) \in C_0$. Thus 
	$$r(x + y) = r((\alpha + \beta) z) = (\alpha + \beta) r(z) \le \alpha + \beta = r(x) + r(y).$$
	This completes the proof.
\end{proof} 
\begin{remark}
	The converse of Theorem \ref{nls} is trivial. Let $(V, \Vert\cdot\Vert)$ be a normed linear space. Then $B = \lbrace v \in V: \Vert v \Vert \le 1 \rbrace$ is a non-empty, radially compact, absolutely convex set in $V$ with $Lead(B) = \lbrace v \in B: \Vert v \Vert = 1 \rbrace$.
\end{remark} 

\subsection{Centre of a convex set} 
\begin{definition}
	Let $B$ be non-empty convex set in a real vector space $X$. An element $b_0 \in B$ is said to be a \emph{centre} of  $B$, if for each $b \in B$, there exists a (unique) $b' \in B$ such that $b_0 = \frac{1}{2} b + \frac{1}{2} b'$. 
\end{definition} 
Note that if $C$ is an absolutely convex set, then $0$ is a centre of $C$. 
\begin{proposition}\label{17}
	Let $B$ be a non-empty convex set in a real vector space $X$ with a centre $b_0$. Put $B_0 = B - b_0$ and assume that $B_0$ is radially compact. 
	\begin{enumerate}
		\item Then $B_0$ balanced and convex. 
		\item If $x \in Lead(B_0)$, then $- x \in Lead(B_0)$. 
		\item Put $K = \lbrace x + b_0: x \in Lead(B_0) \rbrace$. Then for each $b \in B$ with $b \ne b_0$, there exist a unique pair $b_1, b_2 \in K$ with $\frac{1}{2} b_1 + \frac{1}{2} b_2 = b_0$ and $\frac{1}{2} < \alpha \le 1$ such that $b = \alpha b_1 + (1 - \alpha) b_2$. 
	\end{enumerate}
\end{proposition}
\begin{proof}
	(1). As $B$ is convex, so is $B_0$. Also, $0 \in B_0$ for $b_0 \in B$. Let $x \in B_0$. Then $x = b - b_0$ for some $b \in B$. As $b_0$ is a centre of $B$, there exists a unique $b' \in B$ such that $\frac{1}{2} b + \frac{1}{2} b' = b_0$. Thus $- x = b' - b_0 \in B_0$. Also, for $0 \le \lambda \le 1$, we have $\lambda x = \lambda x + (1 - \lambda) 0 \in B_0$ so that $B_0$ is balanced as well.
	
	(2). First observe that $- x \in B_0$ as $B_0$ is balance. As $x \in Lead(B_0)$, we have $x \ne 0$ so that there exist a unique $x_1 \in Lead(B_0)$ and a unique $\lambda \in (0, 1]$ such that $- x = \lambda x_1$. Then $x = \lambda (- x_1)$. Now, by the definition of a lead point, we get $\lambda = 1$ and consequently, $- x = x_1 \in Lead(B_0)$.
	
	(3). Let $b = x + b_0$ for some $x \in B_0$. As $b \ne b_0$, we have $x \ne 0$. Thus there exist a unique $x_1 \in Lead(B_0)$ and a unique $\lambda \in (0, 1]$ such that $x = \lambda x_1$. Put $\alpha = \frac{1 + \lambda}{2}$. Then $\frac{1}{2} < \alpha \le 1$. For $b_1 = x_1 + b_0$ and $b_2 = - x_1 + b_0$, we have $\frac{1}{2} b_1 + \frac{1}{2} b_2 = b_0$. Also, by (2), we have $b_1, b_2 \in K$. Thus 
	$$b = x + b_0 = (1 - \lambda) b_0 + \lambda b_1 = (1 - \lambda) \left( \frac{1}{2} b_1 + \frac{1}{2} b_2 \right) + \lambda b_1 = \alpha b_1 + (1 - \alpha) b_2.$$ 
	Next, let $b_3, b_4 \in K$ with $\frac{1}{2} b_3 + \frac{1}{2} b_4 = b_0$ and $\beta \in \mathbb{R}$ with $\frac{1}{2} < \beta \le 1$ be such that $b = \beta b_3 + (1 - \beta) b_4$. Let $b_3 = x_2 + b_0$ for some $x_2 \in Lead(B_0)$. Then $b_4 = - x_2 + b_0$ and we have  
	$$x = \beta x_2 + (1 - \beta) (- x_2) = (2 \beta - 1) x_2.$$ 
	Now, by the uniqueness of a lead representation, we get $x_2 = x_1$ and 
	$$2 \beta - 1 = \lambda = 2 \alpha - 1$$ 
	so that $\beta = \alpha$.
\end{proof}
\begin{corollary}
	Let $B$ be a non-empty convex set in a real vector space $X$ and assume that $b_0 \in B$ be a centre of $B$. If $B_0 = B - b_0$ be radially compact, then $b_0$ is the only centre of $B$.
\end{corollary}
\begin{proof}
	Let $b_1 \in B$ be another centre of $B$. Put $u_0 = b_1 - b_0$. Then $u_0 \in B_0$. We show that $u_0 = 0$. As $b_1$ is a centre of $B$ and $b_0 \in B$. There exists $b_0' \in B$ such that $b_1 = \frac 12 b_0 + \frac 12 b_0'$. Thus $2 u_0 = b_0' - b_0 \in B_0$. Assume that $n u_0 \in B_0$ for some $n \in \mathbb{N}$. Then $b_2 := - n u_0 + b_0 \in B$ as $B_0$ is balanced. Thus there exists $c \in B$ such that $b_1 = \frac 12 b_2 + \frac 12 c$. It follows that $(n + 2) u_0 = (c - b_0) \in B_0$. Now, by induction, $n u_0 \in B_0$ for all $n \in \mathbb{N}$. Since $B_0$ is radially compact, we must have $u_0 = 0$. Hence $b_1 = b_0$. 
\end{proof}

\section{Affine functions of a central compact convex set} 

Let $B$ be a compact and convex set in a real locally convex Hausdorff space $X$ with a centre $b_0$ and consider the compact, balanced, convex set $B_0 = B - b_0$. Recall that a function $f: B \to \mathbb{R}$ is called \emph{affine}, if 
$$f(\alpha b_1 + (1 - \alpha) b_2) = \alpha f(b_1) + (1 - \alpha) f(b_2)$$
for all $b_1, b_2 \in B$ and $\alpha \in [0, 1]$.
Let 
$$A(B) = \lbrace f: B \to \mathbb{R}| ~ f ~ \textrm{is affine and continuous} \rbrace.$$
As $B$ is compact and convex, $A(B)$ is an order unit space with the cone 
$$A^+(B) = \lbrace f \in A(B): f(b) \ge 0 ~ \textrm{for all} ~ b \in B \rbrace$$
and the order unit ${\bf 1}: B \to \mathbb{R}$ given by ${\bf 1}(b) = 1$ for all $b \in B$. Also the order unit norm is the sup-norm on $A(B)$ as a closed subspace of $C(B)$. (For a detailed discussion on this topic, please refer to \cite[Chapter II.1]{A71}.) 

In this section, we describe the space of real valued, continuous, affine functions on a central, compact, convex set in a real locally convex Hausdorff space. 
\begin{theorem}\label{centralccs}
	Let $B$ be central compact and convex set in a real locally convex Hausdorff space $X$ with the centre $b_0$ and consider the closed subspace 
	$$A_0(B) := \lbrace f \in A(B): f(b_0) = 0 \rbrace$$ of $A(B)$. Then $A(B)$ is unitally order isomorphic to the order unit space $(A_0(B)^{(\cdot)}, e)$. 
\end{theorem}
We shall prove this result with the help of the following lemmas.
\begin{lemma}\label{19}
	$A^+(B) = \lbrace f \in A(B): \Vert f - f(b_0) {\bf 1} \Vert_{\infty} \le f(b_0) \rbrace$.
\end{lemma}
\begin{proof}
	First, we assume that $f \in A^+(B)$ so that $f(b) \ge 0$ for all $b \in B$. Fix $b \in B$ with $b\ne b_0$. Then by Proposition \ref{17}(3), there exist a unique pair $b_1, b_2 \in K$ with $\frac{1}{2} b_1 + \frac{1}{2} b_2 = b_0$ and a unique $\alpha \in \mathbb{R}$ with $\frac{1}{2} < \alpha \le 1$ such that $b = \alpha b_1 + (1 - \alpha) b_2$. Thus 
	$$f(b) = \alpha f(b_1) + (1 - \alpha) f(b_2) \le \max \lbrace f(b_1), f(b_2) \rbrace.$$ 
	Again, as $\frac{1}{2} b_1 + \frac{1}{2} b_2 = b_0$, we have 
	$$\frac{1}{2} f(b_1) + \frac{1}{2} f(b_2) = f(b_0)$$ 
	so that 
	$$0 \le f(b) \le \max \lbrace f(b_1), f(b_2) \rbrace \le 2 f(b_0).$$ 
	Thus $\vert f(b) - f(b_0) \vert \le f(b_0)$ for all $b \in B$. Therefore, $\Vert f - f(b_0) {\bf 1} \Vert_{\infty} \le f(b_0)$. 
	
	Conversely, we assume that $f \in A(B)$ with $\Vert f - f(b_0) {\bf 1} \Vert_{\infty} \le f(b_0)$. Then for each $b \in B$, we have $\vert f(b) - f(b_0) \vert \le f(b_0)$ or equivalently, 
	$$0 \le f(b) \le \max \lbrace f(b_1), f(b_2) \rbrace \le 2 f(b_0).$$ 
	Thus $f \in A(B)^+$. 
\end{proof}
\begin{lemma}\label{20}
	For each $f \in A(B)$, we have 
	$$\Vert f \Vert_{\infty} = \Vert f - f(b_0) {\bf 1} \Vert_{\infty} + \vert f(b_0) \vert.$$
\end{lemma}
\begin{proof}
	Fix $f \in A(B)$. Since $\Vert \cdot \Vert_{\infty}$ is an order unit norm on $A(B)$, we have 
	$$\Vert f - f(b_0) {\bf 1} \Vert_{\infty} {\bf 1} \pm (f - f(b_0) {\bf 1}) \in A^+(B).$$ 
	As 
	$$\Vert f - f(b_0) {\bf 1} \Vert_{\infty} + \vert f(b_0) \vert = \max \lbrace \Vert f - f(b_0) {\bf 1} \Vert_{\infty} + f(b_0), \Vert f - f(b_0) {\bf 1} \Vert_{\infty} - f(b_0) \rbrace,$$ 
	we obtain that 
	$$(\Vert f - f(b_0) {\bf 1} \Vert_{\infty} + \vert f(b_0) \vert) {\bf 1} \pm f \in A^+(B).$$
	Thus $\Vert f \Vert_{\infty} \le \Vert f - f(b_0) {\bf 1} \Vert_{\infty} + \vert f(b_0) \vert$.
	
	Next, as $\Vert f \Vert_{\infty} {\bf 1} \pm f \in A^+(B)$, by Lemma \ref{19}, we have 
	$$\Vert (\Vert f \Vert_{\infty} {\bf 1} \pm f) - (\Vert f \Vert_{\infty} \pm f(b_0)) {\bf 1} \Vert_{\infty} \le (\Vert f \Vert_{\infty} \pm f(b_0))$$
	or equivalently, $$\Vert f - f(b_0) {\bf 1} \Vert_{\infty} \le (\Vert f \Vert_{\infty} \pm f(b_0)).$$ 
	Thus 
	\begin{eqnarray*}
		\Vert f - f(b_0) {\bf 1} \Vert_{\infty} + \vert f(b_0) \vert &=& \max \lbrace \Vert f - f(b_0) {\bf 1} \Vert_{\infty} + f(b_0), \Vert f - f(b_0) {\bf 1} \Vert_{\infty} - f(b_0) \rbrace \\
		&\le& \Vert f \Vert_{\infty}.
	\end{eqnarray*}
	Hence $\Vert f \Vert_{\infty} = \Vert f - f(b_0) {\bf 1} \Vert_{\infty} + \vert f(b_0) \vert$.
\end{proof} 
\begin{proof}[Proof of Theorem \ref{centralccs}]
	For $f \in A(B)$, we define $\theta(f): B \to \mathbb{R}$ given by $\theta(f)(b) = f(b) - f(b_0)$ for all $b \in B$. Then $\theta(f) = f - f(b_0) {\bf 1} \in A_0(B)$ so that by Lemma \ref{20}, we have 
	$$\Vert \theta(f) \Vert_{\infty} = \Vert f - f(b_0) {\bf 1} \Vert_{\infty} \le \Vert f \Vert_{\infty}.$$
	 Thus the map $\theta: A(B) \to A_0(B)$ given by $\theta(f) = f - f(b_0) {\bf 1}$ for all $f \in A(B)$ is a contractive, linear surjective map. We define $\chi: A(B) \to A_0(B)^{(\cdot)}$ given by $\chi(f) = (\theta(f), f(b_0))$ for all $f \in A(B)$. Then $\chi$ is a unital linear surjection. Also, by Lemmas \ref{19} and \ref{20}, $\chi$ is an order isomorphism.
\end{proof} 
\begin{example}
	Fix $n \in \mathbb{N}$ with $n \ge 2$ and consider the compact convex set 
	$$S_n = \lbrace (x_i) \in \ell_1^n: x_i \ge 0 ~ \textrm{and} ~ \sum_{i=1}^{n} x_i = 1 \rbrace.$$ 
	Then $B_n := \co (S_n \bigcup - S_n)$ is the closed unit ball of the base normed space $\ell_1^n$ and is a compact convex set with the centre $0$. We show that $A(B_n)$ is isometrically isomorphic to $(A(S_n)^{(\cdot)}, e)$. Note that $A(S_n)$ is isometrically isomorphic to $\ell_{\infty}^n$. 
	
	Let $f \in A_0(B_n)$. Then $f(- \alpha) = - f(\alpha)$ for all $\alpha \in B_n$. In fact if $\alpha \in B_n$, then $- \alpha \in B_n$ and we have $0 = \frac 12 \alpha + \frac 12 (- \alpha)$. Thus 
	$$0 = f(0) = \frac 12 f(\alpha) + \frac 12 f(- \alpha)$$ 
	so that $f(- \alpha) = - f(\alpha)$. Put $a_i = f(e_i)$ where $\lbrace e_i \rbrace$ is the standard unit vector basis of $\mathbb{R}^n$. We show that $f((x_i)) = \sum_{i=1}^{n} x_i f(e_i)$ for every $(x_i) \in B_n$. Fix $(x_i) \in B_n$. Then $\sum_{i=1}^{n} \vert x_i \vert \le 1$. Put $\epsilon_i = sign(x_i)$ for $1 \le i \le n$. Then $\epsilon_i = \pm 1$ and $x_i = \epsilon \vert x_i \vert$ for each $i$. Also, we have $\epsilon f(e_i) = f(\epsilon_i e_i)$ for each $i$. Put $\lambda_i = \vert x_i \vert$ for $1 \le i \le n$ and $\lambda_0 = 1 - \sum_{i=1}^{n} \vert x_i \vert$. Then $\lambda_i \ge 0$ for $0 \le i \le n$ with $\sum_{i=0}^{n} \lambda_i = 1$.  Thus as $f$ is affine, we have  
	\begin{eqnarray*}
		\sum_{i=1}^{n} x_i a_i &=& \sum_{i=1}^{n} \vert x_i \vert \epsilon_i f(e_i) \\
		&=& \sum_{i=1}^{n} \vert x_i \vert f(\epsilon_i e_i) \\ 
		&=& \lambda_0 f(0) + \sum_{i=1}^{n} \lambda_i f(\epsilon_i e_i) \\
		&=& f(\lambda_0 0 + \sum_{i=1}^{n} \lambda_i \epsilon_i e_i) \\ 
		&=& f((x_i)).
	\end{eqnarray*} 
	Next, let $f \in A(B_n)$. Define $f_0(x) = f(x) - f(0)$ for all $x \in B_n$. Then $f_0 \in A_0(B_n)$. Thus for any $(x_i) \in B_n$, we have 
	\begin{eqnarray*}
		f((x_i)) &=& f_0((x_i)) + f(0) \\ 
		&=& \sum_{i=1}^{n} x_i f_0(e_i) + f(0) \\ 
		&=& \sum_{i=1}^{n} x_i (f(e_i) - f(0)) + f(0).  
	\end{eqnarray*}
	Therefore, $f \mapsto (f(e_1), \dots, f(e_n))$ is linear isomorphism from $A_0(B_n)$ onto $\mathbb{R}^n$ and $f \mapsto (f(0), f(e_1) - f(0), \dots, f(e_n) - f(0))$ is linear isomorphism from $A(B_n)$ onto $\mathbb{R}^{n+1}$. 
	
	Next, let $f \in A(B_n)$. Then 
	\begin{eqnarray*}
		\Vert f \Vert_{\infty} &=& \sup \lbrace \vert f((x_i)) \vert: (x_i) \in B_n \rbrace \\ 
		&=& \sup \left\lbrace \left\vert \sum_{i=1}^{n} x_i (f(e_i) - f(0)) + f(0) \right\vert: \sum_{i=1}^{n} \vert x_i \vert \le 1 \right\rbrace \\ 
		&=& \sup \left\lbrace \left\vert \sum_{i=1}^{n} x_i (f(e_i) - f(0)) \right\vert + \vert f(0) \vert: \sum_{i=1}^{n} \vert x_i \vert \le 1 \right\rbrace \\ 
		&=& \vert f(0) \vert + \max \lbrace \vert f(e_i) - f(0) \vert: 1 \le i \le n \rbrace \\ 
		&=& \vert f(0) \vert + \Vert (f(e_1) - f(0), \dots, f(e_n) - f(0)) \Vert_{\infty}.
	\end{eqnarray*} 
	Hence $A_0(B_n)$ is isometrically isomorphic to $\ell_{\infty}^n$ and $A(B_n)$ is isometrically isomorphic to $\ell_{\infty}^n \oplus_1 \mathbb{R}$. Since $A(S_n)$ is isometrically isomorphic to $\ell_{\infty}^n$, we get that $A(B_n) \cong A(S_n)^{(\cdot)}$ as order unit spaces. 
\end{example} 

\section{Base with a centre} 

Let $(V, B)$ be a base normed space \cite[Proposition II.1.12]{A71}. Then there exists a unique strictly positive $e \in V^*$ with $\Vert e \Vert = 1$ such that $e(v) = \Vert v \Vert$ if and only if $v \in V^+$. In particular, $B = \lbrace v \in V^+: e(v) = 1 \rbrace$ \cite[Lemma 9.3 and Proposition 9.4]{WN73}. In this section, we describe the base normed space with a central base. 
\begin{theorem}\label{centre}
	Let $(V, B)$ be a base normed space. For a fixed $b_0 \in B$, the following statements are equivalent: 
	\begin{enumerate}[$(i)$]
		\item $b_0$ is a centre of $B$; 
		\item $V^+ = \lbrace v \in V: \Vert v - e(v) b_0 \Vert \le e(v) \rbrace$; 
		\item $B = \lbrace v \in V: e(v) = 1 ~ \textrm{and} ~ \Vert v - b_0 \Vert \le 1 \rbrace$.
	\end{enumerate} 
\end{theorem} 
We use the following result to prove Theorem \ref{centre}. 
\begin{lemma}\label{b1}
	Let $(V, B)$ be a base normed space and assume that $b_0 \in B$ is a centre of $B$. Then for each $v \in V$, we have 
	$$\Vert v \Vert \ge \max \lbrace \Vert v - e(v) b_0 \Vert, \vert e(v) \vert \rbrace.$$
\end{lemma}
\begin{proof}
	Let $v = \lambda b_1 - \mu b_2$ for some $b_1, b_2 \in B$ and $\lambda, \mu \ge 0$ with $\lambda + \mu = \Vert v \Vert$. Then $e(v) = \lambda - \mu$ so that 
	\begin{eqnarray*}
		v - e(v) b_0 &=& \lambda b_1 - \mu b_2 - (\lambda - \mu) b_0 \\ 
		&=& \lambda (b_1 - b_0) - \mu (b_2 - b_0) \\ 
		&=& \lambda \left( \frac{b_1 - b'_1}{2} \right) - \mu \left( \frac{b_2 - b'_2}{2} \right) \\ 
		&=& \frac{1}{2} \left\lbrace (\lambda b_1 - \mu b_2) - (\lambda b'_1 - \mu b'_2) \right\rbrace.
	\end{eqnarray*}
	Thus $\Vert v - e(v) b_0 \Vert \le \lambda + \mu = \Vert v \Vert$. Also, we have  $\vert e(v) \vert = \vert \lambda - \mu \vert \le \lambda + \mu$ which completes the proof. 
\end{proof}
\begin{proof}[Proof of Theorem \ref{centre}] (i) implies (ii): 
	
	Let us assume that $b_0$ is a centre of $B$. If $v \in V^+$, then $e(v) = \Vert v \Vert$. Thus by Lemma \ref{b1}, $\Vert v - e(v) b_0 \Vert \le e(v)$. Now assume that $\Vert v - e(v) b_0 \Vert \le e(v)$ for some $v \in V$. Find $b_1, b_2 \in B$ and $\lambda, \mu \ge 0$ such that $v - e(v) b_0 = \lambda b_1 - \mu b_2$ and $\Vert v - e(v) b_0 \Vert = \lambda + \mu$. Then 
	$$0 = e(v - e(v) b_0) = e(\lambda b_1 - \mu b_2) = \lambda - \mu$$ 
	so that $2 \lambda = 2 \mu = \Vert v - e(v) b_0 \Vert \le e(v)$. Thus 
	$$v = \lambda (b_1 - b_2) + e(v) b_0 = \lambda b_1 + \lambda b'_2 + (e(v) - 2 \lambda) b_0 \in V^+.$$
	Thus $V^+ = \lbrace v \in V: \Vert v - e(v) b_0 \Vert \le e(v) \rbrace$.
	
	(ii) implies (iii):
	
	Next we assume that $V^+ = \lbrace v \in V: \Vert v - e(v) b_0 \Vert \le e(v) \rbrace$. Let $b \in B$. Then $b \in V^+$ and $\Vert b \Vert = 1 = e(b)$. Thus by Lemma \ref{b1}, $\Vert b - b_0 \Vert \le 1$. Put $u = b - b_0$. Then $e(u) = 0$, $\Vert u \Vert \le 1$ and we have $b = u + b_0$. Conversely, let $u \in V$ with $e(u) = 0$ and $\Vert u \Vert \le 1$ and put $b = u + b_0$. Then $e(b) = 1$ and $\Vert b - e(b) b_0 \Vert = \Vert u \Vert \le 1 = e(b)$. Thus by the assumption, $b \in V^+$ so that $\Vert b \Vert = e(b) = 1$, that is, $b \in B$. Therefore, $B = \lbrace v \in V: e(v) = 1 ~ \textrm{and} ~ \Vert v - b_0 \Vert \le 1 \rbrace$.
	
	(iii) implies (i): 
	
	Finnaly assume that $B = \lbrace v \in V: e(v) = 1 ~ \textrm{and} ~ \Vert v - b_0 \Vert \le 1 \rbrace$. Let $b \in B$ and consider $b' := 2 b_0 - b$. Then $e(b') = 1$ and 
	$$b' - e(b') b_0 = 2 b_0 - b - b_0 = b_0 - b.$$
	As $b \in B$, by the assumption, we have $\vert b' - e(b') b_0 \Vert \le 1$ so that $b' \in B$. Thus $b_0$ is a centre of $B$.
\end{proof} 
\begin{corollary}\label{b4}
	Let $(V, B)$ be a base normed space with a centre $b_0$ of $B$. Then for each $v \in V$, we have $\Vert v \Vert = \max \lbrace \Vert v - e(v) b_0 \Vert, \vert e(v) \vert \rbrace$. 
\end{corollary}
\begin{proof}
	If $v \in V^+ \cup - V^+$, then $\Vert v \Vert = \vert e(v) \vert$. Thus by Lemma \ref{b1}, we get $\Vert v \Vert = \max \lbrace \Vert v - e(v) b_0 \Vert, \vert e(v) \vert \rbrace$. Next, let $v \notin V^+ \cup - V^+$. Then by Theorem \ref{centre}, we have $\vert e(v) \vert < \Vert v - e(v) b_0 \Vert$. Put $u = \frac{v - e(v) b_0}{\Vert v - e(v) b_0 \Vert}$. Then $e(u) = 0$ and $\Vert u \Vert = 1$ so that $k = u + b_0 \in B$, by Theorem \ref{centre}. Also then $k' = - u + b_0$. Next, put $2 \lambda = \Vert v - e(v) b_0 \Vert + e(v)$ and $2 \mu = \Vert v - e(v) b_0 \Vert - e(v)$. Then $\lambda, \mu > 0$ with $\lambda + \mu = \Vert v - e(v) b_0 \Vert$ and $\lambda - \mu = e(v)$. Thus 
	$$\lambda k - \mu k' = (\lambda + \mu) u + (\lambda - \mu) b_0 = v - e(v) b_0 + e(v) b_0 = v.$$ 
	Now, it follows that 
	$$\Vert v \Vert \le \lambda + \mu = \Vert v - e(v) b_0 \Vert$$ 
	so that $\Vert v \Vert = \max \lbrace \Vert v - e(v) b_0 \Vert, \vert e(v) \vert \rbrace$.
\end{proof} 
\begin{corollary}\label{u}
	A centre of $B$, if it exists, is unique.
\end{corollary} 
\begin{proof}
	Let $b_1 \in B$ be another centre of $B$. Then by Theorem \ref{centre}, $b_1 = u_1 + b_0$ for some $u_1 \in V$ with $e(u_1) = 0$ and $\Vert u_1 \Vert \le 1$. If $b_1 \ne b_0$, then $u_1 \ne 0$. In which case, $b_2 = - \Vert u_1 \Vert^{-1} u_1 + b_0 \in B$. Since $b_1$ is a centre, by Theorem \ref{centre}, we must have $\Vert b_2 - b_1 \Vert \le 1$. But this is not true as $\Vert b_2 - b_1 \Vert = 1 + \Vert u_1 \Vert > 1$. Hence $b_1 = b_0$. 
\end{proof} 
Combining Theorem \ref{centre} and Corollary \ref{b4}, we may deduce the following
\begin{theorem}\label{0b}
	 A base normed space $(V, B)$ with a centre $b_0 \in B$ is isometrically isomorphic to $\left(V_0^{(\cdot)}, B_0^{(\cdot)}\right)$ where 
	$V_0 = \lbrace v \in V: e(v) = 0 \rbrace$
	and 
	$$B_0^{(\cdot)} = \lbrace (v, 1): v \in V_0 ~ \textrm{with} ~ \Vert v \Vert \le 1 \rbrace.$$ 
	Here $e$ is the order unit of $V^*$ corresponding to the base $B$.
\end{theorem}
\begin{example}
	Consider $V_n = \ell_1^n(\mathbb{R})$. Then $B_n = \lbrace (\alpha_i): \alpha_i  \ge 0 ~ \textrm{and} ~ \sum \alpha_i = 1 \rbrace$ is the base for $V_n$ so that $(V_n, B_n)$ is a base normed space. For $n \ge 3$, $B_n$ can not have a centre. To see this fix $b_0 = (\alpha_i^0) \in B_n$ so that $\sum \alpha_i^0 = 1$. Also, if $\alpha_m^0 = \inf \lbrace \alpha_i^0: 1 \le i \le n \rbrace$, then $\alpha_m^0 \le \frac{1}{n}$. Consider the $m$-th coordinate vector $e_m \in V_n$. Then $e_m \in B_n$ and we have 
	$$\Vert e_m - b_0 \Vert_1 = 1 - \alpha_m^0 + \sum_{i \ne m} \alpha_i^0 = 2 (1 - \alpha_m^0) \ge 2 (1 - \frac{1}{n}) > 1.$$ 
	Thus $b_0$ is not a centre of $B_n$.
\end{example}

\section{Order unit spaces with a central state space} 

We now describe a non-dual version of $A(B)$ where $B$ is a central, compact convex subset of a real locally convex space.
\begin{proposition}\label{dualo}
	Let $(V, e)$ be an order unit space and let $\tau \in S(V)$. Then the following statements are equivalent: 
	\begin{enumerate}
		\item $\tau$ is a centre of $S(V)$;
		\item $v \le 2 \tau(v) e$ for every $v \in V^+$;
		\item $V^+ = \lbrace v \in V: \Vert v - \tau(v) e \Vert \le \tau(v) \rbrace$.
	\end{enumerate}
	In this case, we say that $\tau$ is a central state.
\end{proposition} 
\begin{proof}
	(1) implies (2): We assume that $\tau$ is a centre of $S(V)$. Let $v \in V^+$. If $f \in S(V)$, then by assumption, $2 \tau - f \in S(V)$. Thus $f(2 \tau(v) e -v) = 2 \tau(v) - f(v) \ge 0$ for all $f \in S(V)$ so that $v \le 2 \tau(v) e$. 
	
	(2) implies (1): Now, we assume that $v \le 2 \tau(v) e$ for each $v \in V^+$. Let $f \in S(V)$. Put $f' = 2 \tau - f$. Then $f'$ is linear and $f'(e) = 1$. We show that it is positive. Let $v \in V^+$. Then by assumption, $v \le 2 \tau(v) e$. As $f \ge 0$, we have $f'(v) = 2 \tau(v) - f(v) = f(2 \tau(v) e - v) \ge 0$. Thus $f' \in S(V)$. 
	
	(2) implies (3): We again assume that $v \le 2 \tau(v) e$ for each $v \in V^+$. Thus if $v \in V^+$, then $v \le 2 \tau(v) e$, or equivalently, $\tau(v) e \pm (v - \tau(v) e) \in V^+$. Therefore, $\Vert v - \tau(v) e \Vert \le \tau(v)$. Conversely, let $v \in V$ with $\Vert v - \tau(v) e \Vert \le \tau(v)$. Then $\tau(v) e \pm (v - \tau(v) e) \in V^+$ so that $2 \tau(v) e - v \in V^+$. Thus $V^+ = \lbrace v \in V: \Vert v - \tau(v) e \Vert \le \tau(v) \rbrace$. 
	
	(3) implies (2): Finally we assume that $V^+ = \lbrace v \in V: \Vert v - \tau(v) e \Vert \le \tau(v) \rbrace$. Then for $v \in V^+$ we have $\Vert v - \tau(v) e \Vert \le \tau(v)$. Thus $\tau(v) e \pm (v - \tau(v) e) \in V^+$. In particular, $v \le 2 \tau(v) e$.
\end{proof} 
The norm condition on positive elements of $V$ is somewhat minimal as we see in the next result. 
\begin{proposition}\label{orderq}
	Let $(V, e)$ be an order unit space with a central state $\tau \in S(V)$. Then for each $v \in V$, we have $\Vert v - \tau(v) e \Vert \le \Vert v - \alpha e \Vert$ for all $\alpha \in \mathbb{R}$.
\end{proposition}
\begin{proof}
	Fix $\alpha \in \mathbb{R}$. Then 
	$$\vert f(v) - \alpha \vert = \vert f(v - \alpha e) \vert \le \Vert v - \alpha e \Vert$$ 
	for every $f \in S(V)$. By Proposition \ref{dualo}, given $f \in S(V)$, there exists a unique $f' \in S(V)$ such that $f + f' = 2 \tau$. Thus 
	$$\vert f(v) - f'(v) \vert \le \vert f(v) - \alpha \vert + \vert f'(v) - \alpha \vert \le 2 \Vert v - \alpha e \Vert.$$ 
	Since $f - f' = 2(f - \tau)$, we get 
	$$\vert f(v - \tau(v) e) \vert = \vert f(v) - \tau(v) \vert \le \Vert v - \alpha e \Vert$$ for all $f \in S(V)$. Hence $\Vert v - \tau(v) e \Vert \le \Vert v - \alpha e \Vert$ for all $\alpha \in \mathbb{R}$. 
\end{proof}
\begin{proposition}\label{spaceo}
	Let $(V, e)$ be an order unit space with a central state $\tau \in S(V)$. Then for each $v \in V$, we have 
		$$\Vert v \Vert = \Vert v - \tau(v) e \Vert + \vert \tau(v) \vert.$$
\end{proposition}
\begin{proof} 
	Fix $v \in V$. Then $\Vert v \Vert \le \Vert v - \tau(v) e \Vert + \vert \tau(v) \vert$. Also as $\Vert v \Vert e \pm v \in V^+$, by (1), we have 
	$$\Vert (\Vert v \Vert e \pm v) - \tau(\Vert v \Vert e \pm v) e \Vert \le \tau(\Vert v \Vert e \pm v).$$ 
	In other words, $\Vert v - \tau(v) e \Vert \le \Vert v \Vert \pm \tau(v)$ so that $\Vert v - \tau(v) e \Vert + \vert \tau(v) \vert \le \Vert v \Vert$. Hence $\Vert v \Vert = \Vert v - \tau(v) e \Vert + \vert \tau(v) \vert$. 
\end{proof} 
\begin{remark}
	Let $(V, e)$ be an order unit space with a central state $\tau \in S(V)$. Then $\tau$ is unique and we have 
	\begin{enumerate} 
		\item $V^{*+} = \lbrace f \in V^*: \Vert f - f(e) \tau \Vert \le f(e) \rbrace = \lbrace f + \alpha \tau: f(e) = 0 ~ \textrm{and} ~ \Vert f \Vert \le \alpha \rbrace$; 
		\item $S(V) = \lbrace f_0 + \tau: f_0 \in V^*, f_0(e) = 0 ~ \textrm{and} ~ \Vert f_0 \Vert \le 1 \rbrace$; and 
		\item For each $f \in V^*$, we have $\Vert f \Vert = \max \lbrace \Vert f - f(e) \tau \Vert, \vert f(e) \vert \rbrace$. 
	\end{enumerate} 
\end{remark} 

\begin{proposition}\label{orders}
	Let $(V, e)$ be an order unit space with a central state $\tau \in S(V)$. Then the Banach dual of $V_0 = \lbrace v \in V: \tau(v) = 0 \rbrace$ is isometrically isomorphic to $V_0^* := \lbrace f \in V^*: f(e) = 0 \rbrace$.
\end{proposition}
\begin{proof}
	Consider the mapping $\chi: V_0^* \to (V_0)^*$ given by $f \mapsto f|_{V_0}$. Then $\chi$ is linear. Also for $f \in V_0^*$ we have 
	\begin{eqnarray*}
		\Vert f \Vert &=& \sup \lbrace \vert f(v) \vert: v \in V ~ \textrm{with} ~ \Vert v \Vert \le 1 \rbrace \\ 
		&=& \sup \lbrace \vert f(v - \tau(v) e) \vert: v \in V ~ \textrm{with} ~ \Vert v - \tau(v) e \Vert + \vert \tau(v) \vert \le 1 \rbrace \\ 
		&=& \sup \lbrace \vert f(v) \vert: v \in V_0 ~ \textrm{with} ~ \Vert v \Vert \le 1 \rbrace.
	\end{eqnarray*}
	Thus $\chi$ is an isometry. Further, for $f_0 \in (V_0)^*$, we define $f: V \to \mathbb{R}$ given by $f(v) = f_0(v - \tau(v) e)$ for all $v \in V$. Then $f \in V_0^*$ with $\chi(f) = f_0$.
\end{proof}
\begin{theorem}\label{orderq1}
	Let $(V, e)$ be an order unit space with a central state $\tau \in S(V)$. Then $V/{\mathbb{R} e}$ is isometrically isomorphic to $V_0 := \lbrace v \in V: \tau(v) = 0 \rbrace$ and $V$ is unitally and isometrically order isomorphic to $V_0 \oplus_1 \mathbb{R}$.
\end{theorem}
\begin{proof}
	It follows from Lemma \ref{orderq} that $\Vert v \Vert = \Vert v + \mathbb{R} e \Vert$ for all $v \in V_0$. Thus $v \mapsto v + \mathbb{R} e$ determines a linear isometric mapping from $V_0$ into $V/{\mathbb{R} e}$. Further, $v - \tau(v)$ is a pull back of $X \in V/{\mathbb{R} e}$ whenever $X = v + \mathbb{R} e$. Thus $V_0$ is isometrically isomorphic to $V/{\mathbb{R} e}$. Next, with the help of Proposition \ref{spaceo}(2), we observe that $v \mapsto (v - \tau(v) e, \tau(v))$ determines a surjective linear isometry such that $e \mapsto (0, 1)$. Similarly, Proposition \ref{spaceo}(1) determines that this mapping is an order isomorphism.
\end{proof} 
\subsection{Strict convexity} 

\begin{definition}
	Let $B_0$ be a compact, balanced, convex subset of a real, locally convex, Hausdorff space $X$. We say that $B_0$ is said to satisfy \emph{Property (S)} with respect to $X^*$, if for each $x_0 \in Lead(B_0)$ there exists at most one $f \in X^*$ such that 
	$$\sup \lbrace \vert f(x) \vert: x \in B_0 \rbrace = f(x_0) = 1.$$
\end{definition} 
\begin{remark}
	The notion of Property (S) is motivated by a characterization of strictly convex spaces by M. G. Krein in 1938.  However, the same result was also proved independently by A. F. Ruston in 1949 \cite[Theorem 1]{R49}. (Please refer to \cite[5.5.1.1]{P07} for a detailed discussion on it.) 
\end{remark}

\begin{definition}\label{traciald}
	Let $(V, e)$ be an order unit space with a central trace $\tau \in S(V)$. We say that $V$ is \emph{tracial}, if $S(V)_0$ satisfies Property (S). 
\end{definition}
Theorem \ref{orderq1} can now be restated as follows. 
\begin{corollary}\label{tracialt}
	Let $V_0$ be a strictly convex real normed linear space. Then $V_0^{(\infty)}$ is precisely a tracial order unit space.
\end{corollary}

\section{Central absolutely base normed space}

In \cite{K21}, we described the absolute value that naturally arises out of the order structure in $V^{(\cdot)}$. We studied the corresponding absolute order unit space and showed that it is a generalization of a spin factor. In this section we shall discuss the description of the canonical absolute value in $V^{(\cdot)}$ in the context of the corresponding base $B^{(\cdot)}$.

Let us recall the notion of absolutely ordered spaces introduced in \cite{K18}. (See also, \cite{K16}.)
\begin{definition}\cite[Definition 3.4]{K18}\label{0d}
	Let $(U, U^+)$ be a real ordered vector space and let $\vert\cdot\vert: U \to U^+$ satisfy the following conditions:               
	\begin{enumerate}
		\item $\vert v \vert = v$ if $v \in U^+$;
		\item $\vert v \vert \pm v \in U^+$ for all $v \in U$;
		\item $\vert k v \vert = \vert k \vert \vert v \vert$ for all $v \in U$ and $k \in \mathbb{R}$; 
		\item If $u, v, w \in U$ with $\vert u - v \vert = u + v$ and $\vert u - w \vert = u + w,$ then $\vert u - \vert v \pm w \vert \vert = u + \vert v \pm w \vert$; 
		\item If $u, v$ and $w \in U$ with $\vert u - v \vert = u + v$ and $0 \le w \le v,$ then $\vert u - w \vert = u + w$.
	\end{enumerate}                 
	Then $(U, U^+, \vert\cdot\vert)$ is called an \emph{absolutely ordered space}. 
\end{definition}
Let $(V, B)$ be a base normed space with a centre $b_0$ of $B$. Consider the set $K = \lbrace u + b_0: e(u) = 0 ~ \textrm{and} ~ \Vert u \Vert = 1 \rbrace \subset B$. We note that if $k \in K$, then $k' \in K$. In fact, $k' = - u + b_0$ whenever $k = u + b_0$. The following results show that $K$ determines an absolute value in $V$. 
\begin{lemma}\label{b6}
	Let $(V, B)$ be a base normed space with a centre $b_0$ of $B$ and let $b \in B$ with $b \ne b_0$. Then there exist a unique $k \in K$ and a unique $\frac{1}{2} < \alpha \le 1$ such that $b = \alpha k + (1 - \alpha) k'$. 
\end{lemma}
\begin{proof}
	As $b \in B$ with $b \ne b_0$, we have $0 < \Vert b - b_0 \Vert \le 1$. Put $\alpha = \frac{1 + \Vert b - b_0 \Vert}{2}$. Then $\frac{1}{2} < \alpha \le 1$. Put $u_0 = \frac{b - b_0}{\Vert b - b_0 \Vert}$. Then $e(u_0) = 0$ with $\Vert u_0 \Vert = 1$. Thus $k = u_0 + b_0 \in K$ with $k' = - u_0 + b_0$ and we have 
	$$\alpha k + (1 - \alpha) k' = (2 \alpha - 1) u_0 + b_0 = \Vert b - b_0 \Vert \left( \frac{b - b_0}{\Vert b - b_0 \Vert} \right) + b_0 = b.$$ 
	Next, let $b = \beta k_1 + (1 - \beta) k'_1$ for some $k_1 \in K$ and $\frac{1}{2} < \beta \le 1$. Then there exists $u_1 \in V$ with $e(u_1) = 0$ and $\Vert u_1 \Vert = 1$ such that $k_1 = u_1 + b_0$ and $k'_1 = - u_1 + b_0$. Then as above $b = (2 \beta - 1) u_1 + b_0$ so that $(2 \alpha - 1) u_0 = (2 \beta - 1) u_1$. Since $\frac{1}{2} < \alpha \le 1$ and $\frac{1}{2} < \beta \le 1$, we have $2 \alpha > 1$ and $2 \beta > 1$. Now, as $\Vert u_0 \Vert = 1 = \Vert u_1 \Vert$, we deduce that $2 \alpha - 1 = 2 \beta - 1$. Thus we have $\alpha = \beta$ so that $u_0 = u_1$,
\end{proof}
\begin{theorem}\label{b7}
	Let $(V, B)$ be a base normed space with a centre $b_0$ of $B$ and let $v \in V \setminus \mathbb{R} b_0$. Then there exist $k \in K$ and $\alpha, \beta \in \mathbb{R}$ such that $v = \alpha k + \beta k'$. This representation is unique in the following sense: if $v = \lambda k_1 + \mu k_1'$ for some $k_1 \in K$ and $\lambda, \mu \in \mathbb{R}$, then either $k_1 = k$ with $\lambda = \alpha$ and $\mu = \beta$ or $k_1 = k'$ with $\lambda = \beta$ and $\mu = \alpha$.
\end{theorem} 
\begin{proof}
	First consider $v \in V^+ \setminus \mathbb{R} b_0$. Then $v = \Vert v \Vert b$ for $b := \Vert v \Vert^{-1} v \in B$ such that $b \ne b_0$. Thus by Lemma \ref{b6}, there exists a unique $k \in K$ and $\frac{1}{2} < \alpha_0 \le 1$ such that $b = \alpha_0 k + (1 - \alpha_0) k'$. Therefore, $v = \alpha k + \beta k'$ where $\alpha = \Vert v \Vert \alpha_0$ and $\beta = \Vert v \Vert (1 - \alpha_0)$. A similar proof works for $v \in - V^+ \setminus \mathbb{R} b_0$.
	
	Now, let $v \notin V^+ \cup - V^+$. Then $\vert e(v) \vert < \Vert v - e(v) b_0 \Vert$. Put $u = \frac{v - e(v) b_0}{\Vert v - e(v) b_0 \Vert}$. Then $e(u) = 0$ and $\Vert u \Vert = 1$ so that $k = u + b_0 \in K$ with $k' = - u + b_0$. First, we assume that $e(v) \ge 0$. Put $2 \alpha = e(v) + \Vert v - e(v) b_0 \Vert$ and $2 \beta = e(v) - \Vert v - e(v) b_0 \Vert$, we have $\alpha > 0$, $\beta < 0$ with $\alpha - \beta = \Vert v - e(v) b_0 \Vert$ and $\alpha + \beta = e(v)$. Thus 
	$$\alpha k + \beta k' = (\alpha - \beta) u + (\alpha + \beta) b_0 = v - e(v) b_0 + e(v) b_0 = v.$$ 
	
	Next, let $v = \lambda k_1 + \mu k'_1$ for some $k_1 \in K$ and $\lambda, \mu \in \mathbb{R}$. Then $e(v) = \lambda + \mu$. Let $k_1 = u_1 + b_0$ for some $u_1 \in V$ with $e(u_1) = 0$ and $\Vert u_1 \Vert = 1$. Then $k' = - u_1 + b_0$ and we have $v = (\lambda - \mu) u_1 + e(v) b_0$. Replacing $u_1$ by $- u_1$, if required, we get that $\lambda \ge \mu$. Thus $\Vert v - e(v) b_0 \Vert = \lambda - \mu$. Now it follows that $\lambda = \alpha$ and $\mu = \beta$ so that $u_1 = u_0$ and consequently, $k = k_1$. 
	
	When $e(v) < 0$, we interchange $\alpha$ and $\beta$ and replace $k$ by $k'$.
\end{proof}

\begin{remark}
	Let us call the unique decomposition $v = \alpha k + \beta k'$ with $k \in K$ and $\vert \alpha \vert \ge \vert \beta \vert$ as the \emph{$K$-decomposition} of $v \in V \setminus \mathbb{R} b_0$. Note that for $v = \lambda b_0$, we have $v = \alpha (k + k')$ for all $k \in K$ where $2 \alpha = \lambda$. Thus, given $v \in V$ and a decomposition $v = \alpha k + \beta k'$ with $\vert \alpha \vert \ge \vert \beta \vert$, we obtain $\vert v \vert = \vert \alpha \vert k + \vert \beta \vert k' \in V^+$ which is uniquely determined by $v$. This defines a mapping $\vert\cdot\vert: V \to V^+$. 
\end{remark}
\begin{proposition}\label{abs}
	Let $(V, B)$ be a base normed space with a centre $b_0$ of $B$. Then the mapping $\vert\cdot\vert: V \to V^+$ satisfies the following properties:
	\begin{enumerate}
		\item $\vert v \vert = v$ if $v \in V^+$;
		\item $\vert v \vert \pm v \in V^+$ for all $v \in V$;
		\item $\vert \lambda v \vert = \vert \lambda \vert \vert v \vert$ for all $v \in V$ and $\lambda \in \mathbb{R}$; 
		\item If $u, v, w \in V$ with $\vert u - v \vert = u + v$ and $\vert u - w \vert = u + w$, then $\vert u - \vert v \pm w \vert \vert = u + \vert v \pm w \vert$. 
	\end{enumerate}   
\end{proposition}
\begin{proof}
	Let $v \in V$ and consider its $K$-decomposition $v = \alpha k + \beta k'$ where $k \in K$ and $\vert \alpha \vert \ge \vert \beta \vert$. Then $\vert v \vert := \vert \alpha \vert k + \vert \beta \vert k'$. As $k, k' \in V^+$, verification of (1), (2) and (3) is straightforward. To prove (4), let $u, v, w \in V$ with $\vert u - v \vert = u + v$ and $\vert u - w \vert = u + w$. Then $u, v, w \in V^+$. If $u - v \in V^+$, we have $u + v = \vert u - v \vert = u - v$ so that $v = 0$. Thus 
	$$\vert u - \vert v \pm w \vert = \vert u - w \vert= u + w =  u + \vert v \pm w \vert.$$ 
	So we may assume that $u - v, u - w \notin V^+ \bigcup - V^+$. Then $u \ne 0$, $v \ne 0$ and $w \ne 0$. Consider the (unique) $K$-decomposition $u - v = \lambda_1 k_1 + \mu_1 k_1'$ for some $k_1 \in K$ and $\vert \lambda_1 \ge \vert \mu_1 \vert$. Then $\vert u - v \vert = \vert \lambda_1 \vert k_1 + \vert \mu_1 \vert k_1'$. Since $\vert u - v \vert = u + v$, we get that $u = \lambda_1 k_1$ and $v = - \mu_1 k_1'$. Thus $\lambda_1 > 0$ and $\mu_1 < 0$. Again, as $\vert u - w \vert = u + w$, we get that $u = \lambda_2 k_2$ and $w = - \mu_2 k_2'$ for some $k_2 \in K$ and $\vert \lambda_2 \ge \vert \mu_2 \vert$ with $\lambda_2 > 0$ and $\mu_2 < 0$. Thus $k_1 = k_2$ and $\lambda_1 = \lambda_2$. Therefore, 
	$$\vert u - \vert v \pm w \vert \vert = \vert \lambda_1 k_1 - \vert \mu_1 \pm \mu_2 \vert k_1' \vert = \lambda_1 k_1 + \vert \mu_1 \pm \mu_2 \vert k_1' = u + \vert v \pm w \vert.$$
\end{proof} 
\begin{remark}
	It follows from the proof of Proposition \ref{abs} that $\vert u - v \vert = u + v$ if and only if there exist $k \in K$ and positive real numbers $\alpha$ and $\beta$ such that $u = \alpha k$ and $v = \beta k'$.
\end{remark} 
\begin{theorem}\label{bstrict} 
	Let $(V, B)$ be a base normed space with a centre $b_0$ of $B$. Consider 
	$$V_0 = \lbrace v \in V: e(v) = 0 \rbrace = \lbrace v - e(v) b_0: v \in V \rbrace.$$
	Then $V_0$ is closed subspace of $V$ and the following statements are equivalent:
	\begin{enumerate}
		\item $K = ext(B)$;
		\item $V_0$ is strictly convex; and 
		\item For all $u, v, w \in V^+$ with $\vert u - v \vert = u + v$ and $0 \le w \le v$ we have $\vert u - w \vert = u + w$.
	\end{enumerate}
\end{theorem}
\begin{proof}
	(1) implies (2): First we assume that $K = ext(B)$. Let $u_0, u_1 \in V_0$ be such that $u_0 \ne u_1$ and $\Vert u_0 \Vert = 1 = \Vert u_1 \Vert$ and assume to the contrary that $\Vert u_{\alpha} \Vert = 1$ where $u_{\alpha} = \alpha u_1 + (1 - \alpha) u_0$ for $0 < \alpha < 1$. Put $k_0 = u_0 + b_0$, $k_1 = u_1 + b_0$ and $k_{\alpha} = u_{\alpha} + b_0$. Then, by construction, $k_0, k_1, k_{\alpha} \in K$ and we have $k_{\alpha} = \alpha k_1 + (1 - \alpha) k_0$. Since $K = Ext(B)$, we deduce that $k_0 = k_1 = k_{\alpha}$. This leads to a contradiction $u_0 = u_1$. Thus $V_0$ must be strictly convex. 
	
	(2) implies (3): Next we assume that $V_0$ is strictly convex. Consider $u, v, w \in V^+$ with $\vert u - v \vert = u + v$ and $0 \le w \le v$. If $u - v \in V^+$, then $v = 0$ so that $w = 0$ and we have $\vert u - w \vert = u + w$. If $v - u \in V^+$, then $u = 0$ so that $\vert u - w \vert = w = u + w$. Thus we may assume that $u - v \notin V^+ \bigcup - V^+$. Then $u \ne 0$ and $v \ne 0$. We may also assume that $w \ne 0$. Since $\vert u - v \vert = u + v$, following the proof of Proposition \ref{abs}, we can find $k \in K$ and $\alpha \ge \beta > 0$ such that $u = \alpha k$ and $v = \beta k'$. Since $B$ is a base for $V^+$ and since $w \in V^+$ with $w \ne 0$, there exist a unique $b \in B$ and $\lambda > 0$ such that $w = \lambda b$. By Lemma \ref{b6}, we have $b = \gamma k_1 + (1 - \gamma) k_1'$ for some $k_1 \in K$ and $\frac 12 \le \gamma \le 1$. Since $k, k_1 \in K$, there exist $x, x_1 \in V_0$ with $\Vert x \Vert = 1 = \Vert x_1 \Vert$ such that $k = x + b_0$ and $k_1 = x_1 + b_0$. Then $k' = - x + b_0$ and $k_1 = - x_1 + b_0$. As $w \le v$ we have 
	$$\lambda \left( \gamma (x_1 + b_0) + (1 - \gamma) (- x_1 + b_0) \right) \le \beta (- x + b_0)$$ 
	so that $- \left( \beta x + \lambda (2 \gamma - 1) x_1 \right) + (\beta - \lambda) b_0 \in V^+$. Now, as $\frac 12 \le \gamma \le 1$ we have $0 \le 2 \gamma - 1 \le 1$. Thus by Theorem \ref{centre}, we get 
	$$\beta - \lambda \le \beta - (2 \gamma - 1) \lambda = \Vert \beta x \Vert - \Vert (2 \gamma - 1) \lambda x_1 \Vert \le \Vert \beta x + \lambda (2 \gamma - 1) x_1 \Vert \le \beta - \lambda.$$ 
	Therefore, we have $\gamma = 1$ and $\Vert \beta x + \lambda x_1 \Vert = \beta - \lambda$. It follows that $b = k_1$ so that $w = \lambda k_1$. Since $V_0$ is strictly convex, we conclude that 
	$$\Vert \beta x + \lambda x_1 \Vert^{-1} (\beta x + \lambda x_1) = \Vert \lambda x_1 \Vert^{-1} (- \lambda x_1),$$ 
	or equivalently, $\beta x + \lambda x_1 = - (\beta - \lambda) x_1$ as $\Vert x \Vert = \Vert x_1 \Vert$. Thus $x_1 = - x$ so that $k_1 = k'$ and $w = \lambda k'$. Now 
	$$\vert u - w \vert = \vert \alpha k - \lambda k' \vert \alpha k + \lambda k' = u + w.$$
	
	(3) implies (1): Finally, we assume that (3) holds. We show that $K = ext(B)$. Let $b \in B$ such that $b \notin K$. Then there exists $x \in V_0$ with $\Vert x \Vert < 1$ such that $b = x + b_0$. As $b_0 = \frac 12 k + \frac 12 k' \notin ext(B)$,  without any loss of generality we may assume that $x \ne 0$. Put $x_1 = \Vert x \Vert^{-1} x$. Then $b_1 := x_1 + b_0 \in B$ and we have $x = \Vert x \Vert b_1 + (1 - \Vert x \Vert) b_0$ is a proper convex combination in $B$. Thus $b \notin ext(B)$. Therefore, $ext(B) \subset K$. 
	
	Conversely, let $k \in K$ and assume that $k \notin ext(B)$. Then there are $b, c \in B$ and $\delta \in (0, 1)$ such that $k = \delta b + (1 - \delta) c$. Let $x_0, y, z \in V_0$ with $\Vert x_0 \Vert = 1$, $\Vert y \Vert \le 1$ and $\Vert z \Vert \le 1$ such that $k = x_0 + b_0$, $b = y + b_0$ and $c = z + b_0$. Then $x_0 = \delta y + (1 - \delta) z$. Now 
	$$1 = \Vert x_0 \Vert = \Vert \delta y + (1 - \delta) z \Vert \le \delta \Vert y \Vert + (1 - \delta) \Vert z \Vert \le \delta + (1 - \delta) = 1$$ 
	so that $\Vert y \Vert = 1 = \Vert z \Vert$. Put $u = - x_0 + b_0$, $v = x_0 + b_0$ and $w = \delta (y + b_0)$. Then $u, v, w \in V^+ \setminus \lbrace 0 \rbrace$ with $u \in K$ and $v = u'$. Thus $\vert u - v \vert = u + v$. Also $v - w = (1 - \delta) (z + b_0) \ge 0$ but $\vert u - w \vert \ne u + w$ for $w \ne \gamma v'$ for any positive real number $\gamma$ which contradicts the assumption that the statement (3) holds. Thus $K \subset ext(B)$. This proves (1).
\end{proof}
\begin{definition}
	Let $(V, B)$ be a base normed space with a centre $b_0$ of $B$ and assume that $K := \lbrace k \in B: \Vert k - b_0 \Vert = 1 \rbrace = ext(B)$. Then $(V, B, b_0)$ is said to be a \emph{central} base normed space. 
\end{definition}
Now Theorem \ref{bstrict} yields in the following result.
\begin{corollary}\label{strictb1}
	Let $(V, B, b_0)$ be a central base normed spaces so that $V_0 = \lbrace v \in V: e(v) = 0 \rbrace$ is strictly convex and that $V$ is isometrically order isomorphic to $V_0^{(1)}$. 
\end{corollary}

\thanks{{\bf Acknowledgements:} 
	The author was partially supported by Science and Engineering Research Board, Department of Science and Technology, Government of India sponsored  Mathematical Research Impact Centric Support project (reference no. MTR/2020/000017).}

\end{document}